\expandafter\def\csname input@path\endcsname{%
    {sty/},
    {img/},
    {bib/}
}

\RequirePackage{xparse,expl3}
\documentclass[numbers,compress]{vmsta2}

\volume{0}
\issue{0}
\pubyear{2025}
\articletype{research-article}
\usepackage{graphicx,amsfonts,amsmath,enumitem,tikz}
\usepackage{amssymb}
\usepackage{amscd}
\usepackage{comment}
\usepackage{verbatim}
\usetikzlibrary{arrows,automata,shapes,snakes}
\usepackage{amsfonts}

\def\todist{\stackrel{d}{\to}}

\def\reel{\hbox{{\rm R}\kern-1em\hbox{{\rm I} }}}
\def\relatif{\ \hbox{{\rm Z}\kern-.4em\hbox{\rm Z}}}
\def\nat{\hbox{{\rm N}\kern-1em\hbox{{\rm I} } }}
\def\comp{\hbox{{\rm C}\kern-.55em\hbox{{\rm I} } }}

\def\smallcomp{\hbox{\fiverm C}\kern-.35em{\hbox{\fiverm I}}}

\def\fudge{\mathchoice{}{}{\mkern.5mu}{\mkern.8mu}}
\def\bbc#1#2{{\rm \mkern#2mu\vbar\mkern-#2mu#1}}
\def\bbb#1{{\rm I\mkern-3.5mu #1}} \def\bba#1#2{{\rm #1\mkern-#2mu\fudge
#1}}
\def\bb#1{{\count4=`#1 \advance\count4by-64 \ifcase\count4\or\bba
A{11.5}\or \bbb B\or\bbc C{5}\or\bbb D\or\bbb E\or\bbb F \or\bbc
G{5}\or\bbb H\or \bbb I\or\bbc J{3}\or\bbb K\or\bbb L \or\bbb
M\or\bbb N\or\bbc O{5} \or \bbb P\or\bbc Q{5}\orrrr \bb R\or\bbc
S{4.2}\or\bba T{10.5}\or\bbc U{5}\or    \bba V{12}\or\bba
W{16.5}\or\bba X{11}\or\bba Y{11.7}\or\bba Z{7.5}\fi}}
\def\rat{\hbox{{\rm Q}\kern-.70em\hbox{{\rm I} } }}

\def \P {\bbb P}

\def\BP{{\Bbb P}}
\def\BP{{\Bbb P}}

\def\BE{{\Bbb E}}

\newcounter{theorem}

\newcounter{problemcounter}

\newtheorem{prob}[problemcounter]{Problem}

\newtheorem{thm}{Theorem}

\newcommand{\eq}{\begin{equation}}
\newcommand{\en}{\end{equation}}
\newcommand{\ignore}[1]{}

\newcommand{\be}{\begin{equation}}
\newcommand{\ber}{\begin{eqnarray}}

\def\bbb#1{{\rm I\mkern-3.5mu #1}} \def\bba#1#2{{\rm #1\mkern-#2mu\fudge
#1}}

\newcommand{\ena}{\end{eqnarray}}

\newcommand{\Ss}{S(n_1,\ldots,n_t)}

\begin{document} 
\begin{frontmatter}

\pretitle{Research Article}

\title{A Chinese restaurant process for  multiset permutations}

\begin{aug}
\author{ \inits{D.}\fnms{Dudley}~\snm{Stark}\ead[label=e1]{d.s.stark@qmul.ac.uk}\orcid{0000-0001-8527-3996} }
\address{\institution{Queen Mary, University of London}, \\
Mile End, London E1 4NS \cny{United Kingdom} }
\runauthor{Dudley Stark}
\end{aug}
\begin{abstract}
Multisets are like sets, except that they can contain multiple copies of their elements. 
If there are $n_i$ copies of $i$, $1\leq i\leq t$, in multiset $M_t$, then there are
$\binom{n_1+\cdots+n_t}{n_1,\ldots, n_t}$ possible permutations of $M_t$.
Knuth showed how to factor any multiset permutation into cycles. 
For fixed $n_i$, $i\geq 1$, we show how to adapt the Chinese restaurant process, which generates random permutations on $n$ elements with weighting
$\theta^{\# \, {\rm cycles}}$, $\theta>0$, sequentially for $n=1,2,\ldots$, to the multiset case, where we fix the $n_i$ and build permutations on $M_t$ sequentially for $t=1,2,\ldots$. 
The number of cycles of a multiset permutation chosen uniformly at random,   
i.e.~$\theta=1$, has distribution given by the sum of independent negative hypergeometric distributed random variables. For all $\theta>0$, and under the assumption that $n_i=O(1)$, we show a central limit theorem as $t\to\infty$ for the number of cycles.
\end{abstract}



\end{frontmatter}

\section{Introduction} 

A   growth rule which received considerable attention in the literature is  the following preferential attachment   algorithm on permutations written as the product of cycles known as the Chinese restaurant process (CRP) \cite{ABT, BMMU, Feng, GI, GW, GS, IKMW, CSP, Stark}. 
For $n\geq 1$, given a permutation of  $[n-1]:=\left\{1,2,\ldots, n-1\right\}$  which has been constructed at step $n-1$,
element $n$ is  either appended to the permutation as a new singleton cycle with probability $\frac{\theta}{\theta+n-1}$, or  inserted into a random position within 
the existing permutation with probability equal to $\frac{n-1}{\theta+n-1}$.
Let $S_n$ be the set of permutations on $[n]$.
The  permutation obtained at step $n$  has the Ewens distribution, which is the uniform distribution on $n!$ permutations in the case $\theta=1$, and, in general, has probability
distribution 
$\theta^{|\pi|}/\theta_{(n)}$, $\pi\in S_n$,
where $\theta_{(n)}:=\theta(\theta+1)\cdots (\theta+n-1)$ and $|\pi|$ is the number of cycles in $\pi$.
                                                                   
Let $\Pi_{n}$ be the  random permutation of $[n]$  at the $n$th step of the CRP.
The distribution of $\Pi_{n}$ is invariant under relabelling.
For  different orders the permutations  are consistent, in the sense that $\Pi_{m}$, for  $m<n$,  can be derived from $\Pi_{n}$ by removing elements $m+1,\ldots,n$ from their cycles and  deleting empty cycles if necessary.

The CRP expresses every permutation $\pi$ of $[n]$ uniquely as the product
$$
\pi=(x_{1,1}\ldots x_{1,m_1} y_1)
(x_{2,1}\ldots x_{2,m_2} y_2)\cdots
(x_{\tau,1}\ldots x_{\tau, m_\tau} y_\tau),\quad \tau\geq 1,
$$
where the following two conditions are satisfied:
\begin{enumerate}\label{conds}
\item
$y_1< y_2<\ldots < y_\tau$
\item
$y_i< x_{ij} {\rm \ for \ }1\leq j\leq m_i, 1\leq i\leq \tau$.
\end{enumerate}
We will see in Section~\ref{multi} that there is a similar unique expression for 
permutations of multisets.

The number of cycles in $\Pi_{n}$, denoted by
$K_n$,
has probability generating function (p.g.f.)
$$
{\mathbb E} z^{K_n}=\frac{(\theta z)_n}{(\theta)_n},
$$
which corresponds to the distribution of  the sum of $n$ independent Bernoulli variables 
$K_n=\sum_{i=1}^n I_i$, 
$I_i\sim {\rm Bernoulli}(\theta/(i+\theta-1))$; see \cite{ABT}.
We have $I_i=1$ exactly when the element $i$ starts a new cycle.
As $n\to\infty$,
$$
K_n \sim \theta\log n~~{\rm a.s.},~~~\frac{K_n-\theta\log n}{\sqrt{\theta\log n}}\stackrel{d}{\to} {\rm N}(0,1),
$$
where $\todist$ denotes convergence in distribution.

In this paper, a CRP will be defined for multiset permutations  and derived using the methods in \cite{Stark}. As motivation, we note that multiset permutations are of fundamental interest in combinatorics \cite{Stanley}
and random multiset permutations have been studied previously \cite{Feray}.

\section{Multiset permutations and their factorisations into\\ cycles}\label{multi}

There are at least two different ways of defining cycles of multiset permutations
\cite{Knuth, Stanley}.
Knuth's way \cite{Knuth} will be used in this paper.
In this section, we summarise the discussion and results on multiset permutation factorisation taken from
Section~5.1.2 of
\cite{Knuth} that is needed in this paper.
In particular, we are going to use Theorem~\ref{Knuth} below to factor multiset permutations.
A multiset is like a set except that it can have repetitions of
identical elements, for example
$$
M=\{1,1,1,2,2,3,4,4,4,4\},
$$
where the set $\{1,2,3,4\}$ is given the usual ordering.
A permutation of a multiset is an arrangement of its elements in a row, for example
$$
3\ 1\ 2\ 4\ 4\ 1\ 2\ 4\ 1\ 4.
$$
We may write this in two-line notation as
\eq\label{permexample}
\left(
\begin{array}{c c c c c c c c c c}
1&1&1&2&2&3&4&4&4&4\\
3& 1& 2& 4& 4& 1& 2& 4& 1& 4
\end{array}
\right).
\en
The {\em intercalation product} $\pi_1\perp\pi_2$ of two multiset permutations
$\pi_1$ and $\pi_2$ is defined in \cite{CF,Foata} and is obtained by
\begin{enumerate}
\item
Expressing $\pi_1$ and $\pi_2$ in the two-line notation.
\item
Juxtaposing these two-line notations.
\item 
Sorting the columns into nondecreasing order of the top line, in such a way that
the sorting is stable, in the sense that left-to-right order in the bottom line is preserved
when the corresponding top line elements are equal.
\end{enumerate}
For example,
\begin{eqnarray*}
&&
\left(
\begin{array}{c c c c c c}
1&1&2&3&4\\
3& 1& 4& 1& 2&
\end{array}
\right)
\top
\left(
\begin{array}{c c c c c}
1&2&4&4&4\\
2& 4& 4& 1& 4
\end{array}
\right)\\
&=&
\left(
\begin{array}{c c c c c c c c c c}
1&1&1&2&2&3&4&4&4&4\\
3& 1& 2& 4& 4& 1& 2& 4& 1& 4
\end{array}
\right)
\end{eqnarray*}
For possibly repeated elements $x_1, x_2, \ldots, x_m$, the {\em cycle}
$(x_1,\ldots, x_n)$ stands for the permutation obtained in two line form by sorting the
columns of
$$
\left(
\begin{array}{c c c c }
x_1&x_2&\ldots&x_m\\
x_2& x_3& \ldots& x_1
\end{array}
\right)
$$
in a stable manner. For example, we have
\begin{eqnarray}
&&(4\ 2\ 4\ 4\ 1\ 3\ 1\ 1\ 2\ 4)\label{cycle}\\
&=&
\left(
\begin{array}{c c c c c c c c c c}
4&2&4&4&1&3&1&1&2&4\\
2&4&4&1&3&1&1&2&4&4
\end{array}
\right)\nonumber\\
&=&
\left(
\begin{array}{c c c c c c c c c c}
1&1&1&2&2&3&4&4&4&4\\
3& 1& 2& 4& 4& 1& 2& 4& 1& 4
\end{array}
\right)\nonumber
\end{eqnarray}
and so (\ref{permexample}) is a cycle.
For these general cycles, $(x_1 x_2 \ldots x_n)$ is not always the same as
$(x_2\ldots x_n x_1)$ and it can happen that a cycle can be written as the intercalation product
of other cycles, as we shall see at the end of this section.

Knuth \cite{Knuth} proves the following factorisation of multiset permutations into cycles.
\begin{thm}[\cite{Knuth}]\label{Knuth}
Let the elements of the multiset $M$ be linearly ordered by the relation $<$.
Every permutation $\pi$ of $M$ has a unique representation as the intercalation
\eq\label{pirep}
\pi=(x_{1,1}\ldots x_{1,m_1} y_1)\top
(x_{2,1}\ldots x_{2,m_2} y_2)\top\cdots \top
(x_{\tau,1}\ldots x_{\tau,n_\tau} y_{\tau}),\quad \tau\geq 1,
\en
where the following two conditions are satisfied:
\begin{enumerate}
\item
$y_1\leq y_2\leq\ldots\leq y_{\tau}$
\item
$y_i<x_{i,j} {\rm \ for \ }1\leq j\leq m_i, \, 1\leq i\leq \tau$.
\end{enumerate}
\end{thm}

Note that (\ref{cycle}) is not in the form (\ref{pirep}). The factorisation of (\ref{permexample}) corresponding to Theorem~\ref{Knuth}
is
\begin{eqnarray*}
&&
(3\ 1)\top (1) \top 
(2\ 4\ 2\ 4\ 4\ 1)\top (4)\\
&=&
\left(
\begin{array}{c c}
3&1\\
1&3
\end{array}
\right)
\top
\left(
\begin{array}{c}
1\\
1
\end{array}
\right)
\top
\left(
\begin{array}{c c c c c c}
2&4&2&4&4&1\\
4&2&4&4&1&2
\end{array}
\right)
\top
\left(
\begin{array}{c}
4\\
4
\end{array}
\right)\\
&=&
\left(
\begin{array}{c c c c c c c c c c}
3&1&1&2&4&2&4&4&1&4\\
1&3&1&4&2&4&4&1&2&4
\end{array}
\right)\\
&=&
\left(
\begin{array}{c c c c c c c c c c}
1&1&1&2&2&3&4&4&4&4\\
3& 1& 2& 4& 4& 1& 2& 4& 1& 4
\end{array}
\right).
\end{eqnarray*}

\section{The CRP for multiset permutations}

Given a fixed sequence of positive integers $n_i\geq 1$, $i\geq 1$, we will construct a CRP process on permutations of multisets indexed by integers $t\geq 0$. We define  $M(n_1,\ldots,n_t)$
 to be the multiset which contains $n_i$ copies of $i$ for 
$i\in [t]$, $t\geq 1$, and $\emptyset$ for $t=0$. We give the set of elements $[t]$ of $M(n_1,\ldots,n_t)$ the usual ordering. 
The set of permutations of 
 $M(n_1,\ldots,n_t)$ is denoted by $\Ss$ and is of size
 $$
|\Ss|=
\binom{N_t}{n_1,\ldots,n_t},
$$
where $N_t=n_1+\cdots+n_t$. 
Note that 
$
S_n=S(\underbrace{{1,1,\ldots,1}}_{\text{$n$ times}}).
$
The random multiset permutation $\tilde\Pi_t$ of our process at the $t$th step is an element of
$\Ss$, represented in accordance with (\ref{pirep}).
At step $t\geq 1$, there are $n_{t}$ copies of $t$ either to be inserted in the existing permutation
$\tilde\Pi_{t-1}$ or to be put into singleton cycles. (It is not possible to put more than one copy of $t$ in a new cycle because of condition 2 of Theorem~\ref{Knuth}.)
Suppose that $k_t\in \{0\}\cup [n_t]$ of the $n_t$ copies of $t$ start $k_t$ new singleton cycles containing $t$, and the $n_t-k_t$ other copies of $t$ are inserted into the existing permutation. 
The $n_t-k_t$ copies of $t$ can be inserted to the left of any of the $y_i$ or $x_{ij}$ in (\ref{pirep}). A `stars and bars' argument shows that the number of ways of inserting $n_t-k_t$ copies of $t$ to the left of an element in the multiset $S(n_1,\ldots, n_{t-1})$
formatted like (\ref{pirep}) is
\eq\label{waysinsert}
\binom{N_{t-1}-1+n_t-k_t}{n_t-k_t},
\en
recalling that 
$$
\binom{-1}{k}=
\left\{
\begin{array}{l l}
1& {\rm if \ }k=0;\\
0& {\rm if \ }k>1.
\end{array}
\right.
$$

We now define the multiset CRP.
Given $\theta>0$, define
$$
F_t(\theta)=
\sum_{k_t=0}^{n_t}
\binom{N_{t-1}-1+n_t-k_t}{n_t-k_t}
\theta^{k_t}
$$
The special case 
\eq\label{special}
F_t(1)=
\sum_{k_t=0}^{n_t}
\binom{N_{t-1}-1+n_t-k_t}{n_t-k_t}
=\binom{N_{t}}{n_t}.
\en
results from the well known identity
\eq\label{GKPiden}
\sum_{\rho=0}^\nu\binom{\mu+\rho}{\rho}=\binom{\mu+\nu+1}{\nu}, \quad \nu,\mu\geq 0\ {\rm integers};
\en
see page 161 of \cite{GKP}.
Note that (\ref{special}) is the total number of ways of extending the permutation from time $t-1$
to time $t$.
At time $t\geq 1$, $k_t$ new singleton cycles containing $t$ are created and the other $n_t-k_t$ copies are inserted into the existing permutation in a specified way with probability $\theta^{k_t}/F_t(\theta)$. If $n_t=1$ for all $t$, then $F_t(\theta)=\theta+t-1$ and we recover the usual CRP. An induction argument on $t$ shows that all possible
$\pi\in\Ss$ can be constructed in this manner.  This can also be seen from the identity
$$
\binom{N_t}{n_1,\ldots, n_t}
= \prod_{s=1}^t \binom{N_s}{n_s}
$$

Next we will find the distribution of $\tilde\Pi_t$.
\begin{thm}\label{main}
Given $\theta>0$, for all $\pi\in\Ss$ the CRP generates $\pi$ at step $t$ with
probability equal to
$\theta^{|\pi|}/S$,
where
\eq\label{Sdef}
S=
\sum_{\pi\in\Ss}\theta^{|\pi|}=\,\prod_{s=1}^t F_s(\theta).
\en
\end{thm}
\begin{proof}
We will use the method in \cite{Stark}. An {\em arborescence} is a directed, rooted tree with edges oriented in agreement with paths from the root to the leaves. Our arborescence has vertex set $V=\bigcup_{s=0}^t \Pi(n_1,\ldots, n_s)$ and the root is $\emptyset$ (when $s=0$). The leaves of the arborescence are $\Pi(n_1,\ldots, n_t)$. Give each leaf $\pi\in\Pi(n_1,\ldots, n_t)$ the weight $\theta^{|\pi|}$, where $|\pi|$ is the number of cycles in $\pi$. There is a directed edge from
$\pi_1\in \Pi(n_1,\ldots, n_{s-1})$ to $\pi_2\in \Pi(n_1,\ldots, n_{s})$, $s\in [t]$, if and only if,
for some $k_{s}\in\{0\}\cup [n_{s}]$,
$\pi_2$ is obtained from $\pi_1$ by inserting $n_{s}-k_{s}$ copies of $s$ to the left of elements of $\pi_1$, when $\pi_1$ is written as (\ref{Knuth}), and creating $k_{s}$
singleton cycles each containing $s$.
For each $\pi\in V$, define $Q(\pi)$ to be the sum of the weights of all leaves 
$\pi^\prime\in\Pi(n_1,\ldots, n_t)$ for which there is a directed path from $\pi$ to $\pi^\prime$. For $\pi\in\Pi(n_1,\ldots, n_s)$, we have
\eq\label{Qdef}
Q(\pi)=\theta^{|\pi|}\prod_{u=s+1}^t F_u(\theta),
\en
which follows immediately from the interpretation of the binomial coefficient preceding 
(\ref{waysinsert}).
Taking $\pi=\emptyset$ and $s=0$ in (\ref{Qdef}) gives (\ref{Sdef}).
With $\pi_1$ and $\pi_2$ taken as above, we have $|\pi_2|=|\pi_1|+k_s$, 
and therefore
\eq\label{transprobs}
p_{\pi_1,\pi_2}:=\frac{Q(\pi_2)}{Q(\pi_1)}=\frac{\theta^{k_s}}{F_s(\theta)}.
\en
These are the transition probabilities of the CRP.
Theorem~1 of \cite{Stark} states that if we put transition probabilities (\ref{transprobs})
on the edges of the arborescence, and let the leaves be absorbing states, then we obtain a Markov chain whose probability of absorption on a given leaf is proportional to the weight of the leaf. By our choice of weight we are done.
\end{proof}

For  different steps the permutations  are consistent, in the sense that $\tilde\Pi_{s}$, for  $s<t$,  can be derived from $\tilde\Pi_{t}$ by removing elements $s+1,\ldots,t$ from their cycles and  deleting empty cycles if necessary.

\section{The distribution of the number of cycles}

Let $X_t$ denote the number of new cycles added in moving from time $t-1$ to time $t$,  
$t\geq 1$, which is the choice of $k_t$ above. The $X_t$ are independent with distributions given by
\eq\label{thetagenpmf}
\P(X_t=k_t)=
\frac{\theta^{k_t}}{F_t(\theta)}
\binom{N_{t-1}-1+n_t-k_t}{n_t-k_t},\quad k_t\in\{0\}\cup [n_t].
\en

If $\theta=1$, then, by (\ref{special}), we have
\eq\label{theta1pmf}
\P(X_t=k_t)=
\frac{1}{\binom{N_t}{n_t}}
\binom{N_{t-1}-1+n_t-k_t}{n_t-k_t}=
\frac{\binom{n_t}{k_t}}{\binom{N_t}{k_t}}\frac{N_t-n_t}{N_t-k_t}
\en
This equals the probability that starting from an urn with $n_t$ balls representing failures and $N_t-n_t$ balls representing successes, that the first $k_t$ balls pulled from the urn 
without replacement are failures, and the $(k_t+1)$th ball pulled from an urn is a success.
This description shows that 
$
Y_t=X_t+1
$
 has a negative hypergeometric distribution, which we now describe.
The negative hypergeometric distribution \cite{JKK,KK,RS} with parameters $N$, $M$ and $r$ is like the negative binomial distribution, but without replacement. There are $N$ balls in total, $M$ balls representing success and $N-M$ balls representing failures. The hypergeometric distribution gives the probability that the $r$th success happens on the
$\kappa$th draw, $\kappa=r,r+1,\ldots, N$. A negative hypergeometric random variable $Y$ with these parameters has probability mass function
\eq\label{neghyperpmf}
\BP(Y=\kappa)=
\frac{\binom{M}{r-1}\binom{N-M}{\kappa-r}}{\binom{N}{\kappa-1}}\cdot
\frac{M-r+1}{N-\kappa+1}
\en
It is easily checked that (\ref{theta1pmf}) is (\ref{neghyperpmf}) with
$N=N_t$, $M=N_t-n_t=N_{t-1}$, $r=1$ and $\kappa=k+1$.
The expectation of $Y$ is
$$
\BE(Y)=r\frac{N+1}{M+1},
$$
the variance of $Y$ is
$$
{\rm Var}(Y)=
\frac{r(N-M)(N+1)(M+1-r)}{(M+1)^2(M+2)}
$$
and  the third centred moment (see \cite{JKK,KK}) is
$$
\BE([Y-\BE(Y)]^3)={\rm Var}(Y)\cdot\frac{(2N-M+1)(M+1-2r)}{(M+1)(M+3)}.
$$
Therefore,
\eq\label{expXt}
\BE(X_t)=\BE(Y_t)-1=
\frac{N_t+1}{N_{t-1}+1}-1=
\frac{n_t}{N_{t-1}+1},
\en
\eq\label{varXt}
{\rm Var}(X_t)={\rm Var}(Y_t)=
\frac{n_t(N_t+1)N_{t-1}}{(N_{t-1}+1)^2(N_{t-1}+2)}
\en
and
\begin{eqnarray}\label{third}
\BE([X_t-\BE(X_t)]^3)&=&\BE([Y_t-\BE(Y_t)]^3)\nonumber\\
&=&
{\rm Var}(Y_t)\cdot
\frac{(n_t+N_t+1)(N_{t-1}-1)}{(N_{t-1}+1)(N_{t-1}+3)}.
\end{eqnarray}
It follows that the number of cycles $K_t=\sum_{s=1}^t X_t$ has expectation
$$
\BE(K_t)=
\sum_{s=1}^t
\frac{n_s}{N_{s-1}+1}
$$
and variance
$$
{\rm Var}(K_t)=
\sum_{s=1}^t
\frac{n_s(N_s+1)N_{s-1}}{(N_{s-1}+1)^2(N_{s-1}+2)}.
$$

Given sequences $a_t$ and $b_t$, we use the notation $a_t=O(b_t)$ to mean $a_t\leq Cb_t$ for a constant $C>0$; $a_t\asymp b_t$ to mean $C_1b_t\leq a_t\leq C_2b _t$ for constants $C_1>0$, $C_2>0$; and $a_t\sim b_t$ to mean $\lim_{t\to\infty}a_t/b_t = 1$.
We also define $a_t=o(b_t)$ to mean $\lim_{t\to\infty} a_t/b_t=0$.
We will put the bound the $n_t=O(1)$ to give a central limit theorem for $K_t$ for all $\theta$.

\begin{thm}\label{firstmain}
If $\theta=1$ and $n_i=n\geq 1$ for all $i\geq 1$, then
\eq\label{simest}
\BE(K_t)\sim \log t, \quad
{\rm Var}(K_t)\sim \log t.
\en
For general $\theta$, let $n_1, n_2, \ldots, $ be given, $n_i\geq 1$ for all $i$, and suppose 
\eq\label{nibound}
n_i=O(1).
\en
Then,
\eq\label{estimates}
\BE(K_t)\asymp \log t, \quad
{\rm Var}(K_t)\asymp \log t,
\en
and
\eq\label{CLT}
K_n \sim \BE(K_n)~~{\rm a.s.},~~~~~~
\frac{K_t-\BE(K_t)}
{\sqrt{{\rm Var}(K_t)}}
\stackrel{d}{\to} {\rm N}(0,1).
\en
\end{thm}
\begin{proof}
First we prove the theorem for $\theta=1$. 
From (\ref{expXt}), (\ref{varXt}) and (\ref{nibound}) we get $\BE(X_t)\asymp 1/t$ and
${\rm Var}(X_t)\asymp 1/t$, which result in (\ref{estimates}), after which
the Lindeberg-Feller Theorem easily shows the central limit theorem in (\ref{CLT}), because $X_t\leq n_t=O(1)$.
If we impose the condition $n_i=n$, for $i\geq 1$, then we immediately obtain (\ref{simest}).
For $\theta$ general, we reduce to the case $\theta=1$.
This follows from  noting that in (\ref{thetagenpmf}), $\theta^{k_t}\asymp 1$ and $F_t(\theta)\asymp F_t(1)$ uniformly over $t\geq 1$, by (\ref{nibound}), and so
(\ref{estimates})  still holds. 
For the almost sure convergence, modify the proof of (6.6)~Theorem on page 44 of \cite{durrett}.
\end{proof}

 By the following theorem a Central Limit Theorem can be obtained even when $n_t$ is unbounded, which is not the case in Theorem~\ref{firstmain}.
\begin{thm}\label{norm}
Suppose $\theta=1$, $n_t=O(N_{t-1})$.
Then, (\ref{CLT}) holds.
\end{thm}\label{Lap}
\begin{proof}
In order to prove limiting normality, we will verify Lyapounov's Condition in \cite{Bill} with 
$\delta=1$, by
showing that
\eq\label{Lyapanov}
\lim_{t\to\infty}\,
\frac{1}{{\rm Var}(K_t)^{3/2}}
\sum_{s=1}^t \BE([X_s-\BE(X_s)]^3)=0.
\en
The assumption $n_t=O(N_{t-1})$, (\ref{varXt}) and (\ref{third}) produces
$$
\BE([X_t-\BE(X_t)]^3\asymp{\rm Var}(X_t)\asymp n_t/N_{t-1},\quad t\geq 2.
$$
We have the lower bound
\begin{eqnarray*}
\sum_{s=2}^t \frac{n_s}{N_{s-1}}&\geq &
\sum_{s=2}^t \ln \left( 1+\frac{n_s}{N_{s-1}}\right)\\
&=& \sum_{s=2}^t \left(\ln(N_s)-\ln(N_{s-1}\right)\\
&=&\ln(N_t)-\ln(N_1)\to\infty,\\
\end{eqnarray*}
where the convergence to infinity follows from the assumption $n_t\geq 1$ for all $t\geq 1$.
Hence,
$$
{\rm Var}(K_t)\asymp \sum_{s=1}^t \BE([X_s-\BE(X_s)]^3) \asymp
\sum_{s=2}^t \frac{n_s}{N_{s-1}}\to\infty
$$
and (\ref{Lyapanov}) holds.

For the almost sure convergence, modify the proof of (6.6)~Theorem on page 44 of \cite{durrett}
as in the proof of Theorem~\ref{firstmain}.
\end{proof}

An example where Theorem~\ref{norm} holds is obtained by setting $\theta=1$, $n_t=\lceil Ct^\alpha\rceil$, for constants $C>0$, $\alpha>0$, where 
$\lceil x\rceil$ is the smallest integer greater or equal to $x$. Then,
$n_t= Ct^\alpha + O(1)$ and
$N_t= Ct^{\alpha+1}/(\alpha+1) + O(t^\alpha)$ and the assumptions of Theorem~\ref{norm}
hold. The expectation and variance are asymptotically
$$
\BE(K_t)\sim(\alpha+1) \log t, \quad
{\rm Var}(K_t)\sim (\alpha+1)\log t.
$$

Another example for $\theta=1$ is obtained through regular variation \cite{BGT}.  A positive, measurable function $\ell(x)$, defined on $[\eta,\infty)$ for $\eta$ real, 
 is {\em slowly varying} if
\eq\label{slow}
\lim_{x\to\infty}\frac{\ell(\lambda x)}{\ell(x)}=1,\quad \forall \lambda >0.
\en
The Uniform Convergence Theorem \cite{BGT} provides that the convergence in (\ref{slow})
is uniform on compact $\lambda$-sets in $(0,\infty)$.
A positive, measurable function defined on $[\eta,\infty)$,  is {\em regularly varying} with real index $\alpha$ if
$$
\lim_{x\to\infty}\frac{f(\lambda x)}{f(x)}=\lambda^\alpha,\quad \forall \lambda >0.
$$
A regularly varying function with index $\alpha$ can be written as
$$
f(x)=x^\alpha\ell(x),
$$
where $\ell(x)$ is a slowly varying function.
Suppose that $\ell(x)$ is locally bounded in $[\eta,\infty)$, i.e.~bounded on compact 
subsets of $[\eta,\infty)$, and $\alpha>-1$.
Karamata's theorem gives 
$$
\int_\eta^t x^\alpha\ell(x)\,dx\sim t^{\alpha+1}\ell(t)/(\alpha+1),\quad t\to\infty.
$$
We further suppose that $\eta\leq 1$, $\alpha\geq0$, with $\lim_{x\to\infty}\ell(x)=\infty$ in the case $\alpha=0$, and define
$$
n_t=\Big\lceil\int_t^{t+1} x^\alpha\ell(x)\,dx\Big\rceil,\quad t\geq 1.
$$
The Uniform Convergence Theorem implies that 
$$
\sup_{t\leq x\leq t+1}\left|\frac{\ell(x)}{\ell(t)}-1\right|\leq
\sup_{1\leq \lambda \leq 2}\left|\frac{\ell(\lambda t)}{\ell(t)}-1\right|
=o(1),
$$
from which
\begin{eqnarray*}
\int_t^{t+1} x^\alpha\ell(x)\,dx&=&(1+o(1))\ell(t)\int_t^{t+1} x^\alpha\,dx\\
&=&
(1+o(1))\ell(t)(t^\alpha+o(t^\alpha))\\
\end{eqnarray*}
and
$$
n_t\sim t^\alpha\ell(t).
$$
Moreover, Karamata's theorem gives us
$$
N_{t-1}\sim t^{\alpha+1}\ell(t)/(\alpha+1).
$$
Hence, the conditions of Theorem~\ref{norm} hold.

We will now show that $n_t$ may be chosen in such a way that the conclusion of the central limit theorem
does not hold for $K_t$.
Suppose that $\theta=1$ and the $n_t$ are chosen such that 
\eq\label{conda}
N_{t-1}=o(n_t).
\en
By (\ref{expXt}) and (\ref{varXt}),
we have
$$
\BE(X_t)\sim \frac{n_t}{N_{t-1}}, \quad
{\rm Var}(X_t)\sim \left(\frac{n_t}{N_{t-1}}\right)^2.
$$
We further choose the $n_t$ such that
\eq\label{condb}
\sum_{s=1}^{t-1} {\rm Var}(X_s) = o\left({\rm Var}(X_t)\right),
\en

A particular example of such a sequence is 
$$
n_t=\lceil e^{t^3}\rceil=e^{t^3}+O(1),
$$
from which
\begin{eqnarray*}
N_t&=&\sum_{s=1}^t e^{s^3}+O(t)\\
&=&e^{t^3}\left(1+\sum_{s=1}^{t-1} e^{s^3-t^3}\right)+O(t)\\
&=&e^{t^3}\left(1+O\left(te^{(t-1)^3-t^3}\right)\right)+O(t)\\
&=&e^{t^3}(1+o(1)).
\end{eqnarray*}
Clearly, (\ref{conda}) holds and we also have
\begin{eqnarray*}
\sum_{s=1}^{t-1} {\rm Var}(X_s) &=& O\left(\sum_{s=1}^{t-1}e^{2s^3-2(s-1)^3}\right)\\
&=&
O\left(\sum_{s=1}^{t-1}e^{6s^2-6s}\right)\\
&=&
O\left(te^{6(t-1)^2-6(t-1)}\right)\\
&=&
o\left(e^{6t^2-6t}\right)
\end{eqnarray*}
and so (\ref{condb}) holds, as well, because ${\rm Var}(X_t)\sim e^{6t^2-6t+2}$.

Assuming (\ref{conda}) and (\ref{condb}), we have
\eq\label{Kt}
\frac{K_t-\BE(K_t)}{\sqrt{{\rm Var}(K_t)}}=
\frac{X_t-\BE(X_t)}{\sqrt{{\rm Var}(X_t)}}(1+o(1))+\frac{\sum_{s=1}^{t-1} (X_s-\BE(X_s)) }
{\sqrt{{\rm Var}(K_t)}}.
\en
Suppose that 
$$
\frac{K_t-\BE(K_t)}{\sqrt{{\rm Var}(K_t)}} \todist {\rm N}(0,1).
$$
By our assumptions, the variance of the last term of (\ref{Kt})
 is $o(1)$ and so by Slutsky's theorem \cite{durrett}
 \eq\label{conv}
\lim_{t\to\infty}\BP\left(\frac{X_t-\BE(X_t)}{\sqrt{{\rm Var}(K_t)}}\leq 0\right)
=\lim_{t\to\infty}\BP\left(\frac{K_t-\BE(K_t)}{\sqrt{{\rm Var}(X_t)}}\leq 0\right)=1/2,
\en
We will show that (\ref{conv}) is impossible. The c.d.f. of $X_t$ is given by
\begin{eqnarray}
P(X_t\leq x_t)&=&\sum_{k_t=0}^{x_t} \BP(X_t=k_t)\nonumber\\
 &=&\frac{1}{\binom{N_t}{n_t}}\sum_{k_t=0}^{x_t}\binom{N_{t-1}-1+n_t-k_t}{n_t-k_t}\label{step1}\\
 &=&\frac{1}{\binom{N_t}{n_t}}\left( \sum_{j=0}^{n_t}\binom{N_{t-1}-1+j}{j}
 - \sum_{j=0}^{n_t-x_t-1}\binom{N_{t-1}-1+j}{j}\right)\nonumber\\
  &=&\frac{1}{\binom{N_t}{n_t}}\left( \binom{N_t}{n_t}
 - \binom{N_t-x_t-1}{n_t-x_t-1}\right)\label{step2}\\
 &=& 1-\frac{(N_t-x_t-1)!/(n_t-x_t-1)!}{N_t!/n_t!}\nonumber\\
 &\geq&
 1-\left(\frac{N_t-x_t-1}{N_t}\right)^{N_{t-1}},\label{step3}
\end{eqnarray}
where we have used (\ref{theta1pmf}) at (\ref{step1}), (\ref{GKPiden}) at (\ref{step2}),
and $(N_t-x_t-1-i)/(N_t-i)\leq (N_t-x_t-1)/N_t$ for all $0\leq i\leq N_{t-1}-1$ at (\ref{step3}).
We therefore have
\begin{eqnarray*}
\liminf_{t\to\infty}P(X_t\leq \BE(X_t)) &\geq & 1-\lim_{t\to\infty}\left(\frac{N_t-\BE(X_t)-1}{N_t}\right)^{N_{t-1}}\\
&=&
 1-\lim_{t\to\infty}\left(\frac{N_t-n_t(1+o(1)/N_{t-1}-1}{N_t}\right)^{N_{t-1}}\\
 &=&
 1-\lim_{t\to\infty}\left(1-(1+o(1))/N_{t-1}\right)^{N_{t-1}}\\
 &=&1-e^{-1}>1/2,
\end{eqnarray*}
contradicting (\ref{conv}).


\section{Discussion}

Potentially, improvements might be made on these results.
The CRP we have defined places all $n_t$ elements labelled $t$ at the same time.
Placing them sequentially seems like a difficult problem. It would be interesting to get better estimates for $\BE(K_t)$ and
${\rm Var}(K_t)$, especially when $\theta\neq 1$. Perhaps the hypotheses of 
Theorem~\ref{Lap} could be weakened.

In the case of random permutations, there is a Poisson approximation of $K_n$.
The total variation distance between the law 
${\cal L}(K_n)$ of $K_n$ and
the ${\rm Poisson}(\BE(K_n))$ distribution is of order
\begin{eqnarray*}
d_{\rm TV}({\cal L}(K_n),{\rm Poisson}(\BE(K_n)))&:=&
\frac{1}{2}\sum_{k=1}^n \left|
\BP(K_n=k)-\exp\left(-\BE(K_n)\right)\frac{\BE(K_n)^k}{k!}\right|\\
&\asymp& \frac{1}{\log n};
\end{eqnarray*}
see \cite{BH}.
By considering the process of cycle counts $(C_1(n), C_2(n),\ldots, C_n(n))$, where
$C_i(n)$ is the number of cycles of size $i$, we may write $K_n=\sum_{i=1}^n C_i(n)$.
Moreover, \cite{DP} shows that  $B_n(\cdot)\todist B(\cdot)$, where
$$
B_n(t):=
\frac{\sum_{i=1}^{\lfloor n^t \rfloor} C_i(n) - t\log n}{\sqrt{\log n}}, \quad 0\leq t\leq 1,
$$
and $B(t)$ is standard 
Brownian motion.
 If suitable approximations could be made to the process of cycle counts for multiset permutations, then similar progress might be made for the results obtained in this paper when $\theta=1$.

\begin{acknowledgement}[title={Acknowledgments}]
The author thanks the anonymous referee and the Associated Editor for a careful reading of this paper and for helpful suggestions.
\end{acknowledgement}

\bibliographystyle{vmsta2-mathphys}
\bibliography{multiset}

\begin{thebibliography}{99}


\bibitem{ABT} R.~Arratia, A.D.~Barbour and S.~Tavar\'e,
{\em Logarithmic Combinatorial Structures: a Probabilistic Approach},
EMS Monographs in Mathematics, 2003.

\bibitem{BH} Barbour,~A.~D. and Hall,~P.~G.,
On the rate of Poisson convergence, 
{\em Math. Proc. Cambr. Phil. Soc.} {\bf 95} 473--480, 1984.


\bibitem{Bill}
Billingsley,~P., {\em Probability and Measure}, Third Edition, John Wiley \& Sons,
1995.




\bibitem{CF}
Cartier, P. and Foata, D.  Probl\`emes combinatoires de commutation et
r\'earrangements.  {\em Lecture Notes in Math., No. 85} Springer-Verlag {\bf 85} (1969) iv+88 pp.

\bibitem{durrett}
Durrett.~R., {\em Probability: theory and examples}, fourth edition, Cambridge University Press,
2013.




\bibitem{Feng} S. Feng, {\it The Poisson-Dirichlet Distribution and Related Topics}, Springer, 2010.

\bibitem{Feray} V.~F\'	eray, Central limit theorems for patterns in multiset permutations and set partitions. {\em Ann. Appl. Probab.} {\bf 30} (2020),  287--323.


\bibitem{Foata}
Foata, D., \'Etude alg\'ebrique de certains probl\`emes d'analyse combinatoire et
du calcul des probabilit\'es.
{\em Publ.Inst. Statist. Univ Paris} {\bf 14} (1965) 81--241.

\bibitem{GS} A.~Gnedin and D.~Stark.
Random partitions and queues.
{\em Advances in Applied Mathematics}, {\bf 149} 102549, (2023).





\bibitem{GKP} 
Graham, R.~L., Knuth, Donald E. and Patashnik, Oren 
Graham, R.~L., Knuth, D.~E. and Patashnik,~O.,
{\em Concrete mathematics: A foundation for computer science}. Second edition. 
Addison-Wesley Publishing Company, 1994.

\bibitem{JKK}
Johnson, N.~L., Kemp, A.~W., Kotz,~S., {\em Univariate Discrete Distributions}. Third Edition. Wiley. (2005). 

\bibitem{KK} Kemp, C.~D., Kemp, A.~W. Generalized hypergeometric distributions,
{\em J. Roy. Statist. Society, Ser.  B}, {\bf 18}, 202--211.

\bibitem{Knuth}
Knuth, D.~E., {\em The Art of Computer Programming, Volume 3}. Second edition.
Addison-Wesley Publishing Company, 1998.

\bibitem{CSP} Pitman, J., {\em Combinatorial Stochastic Processes}, Lecture Notes in Math. {\bf 1875}, Springer, 2006.

\bibitem{RS} Rohatgi, V. K. and Saleh, A. K. Md., {\em An Introduction to Probability
and Statistics}, Second edition. John Wiley \& Sons, 2001.

\bibitem{Stanley} Stanley,~R.~P., {\em Enumerative combinatorics, volume 1}, second edition, Cambridge University Press, 2012.


\bibitem{Stark}  Stark, D., Markov chains generating random permutations and set partitions. {\em Stochas tic Processes and their Applications}, (2024) 104483.








\end{thebibliography}

\end{document}